\documentclass{amsart}

\usepackage{ae,aecompl}
\usepackage{esint}
\usepackage{graphicx}
\usepackage[all,cmtip]{xy}
\usepackage{amsmath, amscd}
\usepackage{booktabs}
\usepackage{amsthm}
\usepackage{amssymb}
\usepackage{amsfonts}
\usepackage{qsymbols}
\usepackage{latexsym}
\usepackage{mathrsfs}
\usepackage{cite}
\usepackage{color}
\usepackage{url}
\usepackage{enumerate}
\usepackage{undertilde}

\usepackage{verbatim}
\usepackage[draft=false, colorlinks=true]{hyperref}

\allowdisplaybreaks

\newcommand {\bmo}{{\mathrm{bmo}}}
\newcommand {\BMO}{{\mathrm{BMO}}}

\newcommand {\C}{{\mathbb C}}

\newcommand {\ud}{\mathrm{d}}

\newcommand {\veps}{\varepsilon}

\newcommand {\Ell}{L}

\newcommand {\F}{{\mathcal{F}}}

\newcommand {\Hrinf}{\mathcal{H}^{r,\infty}}
\newcommand {\HT}{\mathcal{H}}
\newcommand {\Hp}{\mathcal{H}^{p}_{FIO}(\Rn)}
\newcommand {\Hps}{\mathcal{H}^{s,p}_{FIO}(\Rn)}
\newcommand {\Hpt}{\mathcal{H}^{t,p}_{FIO}(\Rn)}

\newcommand {\rb}{\rangle}
\newcommand {\lb}{{\langle}}
\newcommand {\La}{{\mathcal{L}}}

\newcommand {\N}{{{\mathbb N}}}

\newcommand {\ph}{{\varphi}}

\newcommand {\R}{{\mathbb R}}

\newcommand {\Rn}{{\mathbb{R}^{n}}}

\newcommand {\supp}{{\mathrm{supp}}}

\newcommand {\Sw}{\mathcal{S}}

\newcommand {\w}{{\omega}}

\newcommand {\Z}{{{\mathbb Z}}}

\newcommand {\vanish}[1]{\relax}

\newcommand{\wh}{\widehat}
\newcommand{\wt}{\widetilde}

\DeclareMathOperator{\Real}{Re}
\DeclareMathOperator{\Imag}{Im}

\DeclareFontFamily{U}{mathx}{\hyphenchar\font45}
\DeclareFontShape{U}{mathx}{m}{n}{
      <5> <6> <7> <8> <9> <10>
      <10.95> <12> <14.4> <17.28> <20.74> <24.88>
      mathx10
      }{}
\DeclareSymbolFont{mathx}{U}{mathx}{m}{n}
\DeclareFontSubstitution{U}{mathx}{m}{n}
\DeclareMathAccent{\widecheck}{0}{mathx}{"71}

\newtheorem{theorem}{Theorem}[section]
\newtheorem{lemma}[theorem]{Lemma}
\newtheorem{proposition}[theorem]{Proposition}
\newtheorem{corollary}[theorem]{Corollary}

\theoremstyle{definition}
\newtheorem{definition}[theorem]{Definition}
\newtheorem{remark}[theorem]{Remark}

\numberwithin{equation}{section}
\protected\def\ignorethis#1\endignorethis{}
\let\endignorethis\relax

\title[Rough pseudodifferential operators on Hardy spaces for FIOs II]{Rough pseudodifferential operators on Hardy spaces for Fourier integral operators II}

\author{Jan Rozendaal}
\address{Institute of Mathematics, Polish Academy of Sciences\\
ul.~\'{S}niadeckich 8\\
00-656 Warsaw\\
Poland}
\email{jrozendaal@impan.pl}



\thanks{This research was supported by NCN grant UMO2017/27/B/ST1/00078. The research leading to these results has received funding from the Norwegian Financial Mechanism 2014-2021, grant 2020/37/K/ST1/02765.}

\begin{document}

\begin{abstract}
We obtain improved bounds for pseudodifferential operators with rough symbols on Hardy spaces for Fourier integral operators. The symbols $a(x,\eta)$ are elements of $C^{r}_{*}S^{m}_{1,\delta}$ classes that have limited regularity in the $x$ variable. We show that the associated pseudodifferential operator $a(x,D)$ maps between Sobolev spaces $\Hps$ and $\Hpt$ over the Hardy space for Fourier integral operators $\Hp$. Our main result is that for all $r>0$, $m=0$ and $\delta=1/2$, there exists an interval of $p$ around $2$ such that $a(x,D)$ acts boundedly on $\Hp$.
\end{abstract}

\maketitle

\section{Introduction}\label{sec:intro}

In this article, we further develop a paradifferential calculus adapted to the $L^{p}$ theory of wave equations, by obtaining improved bounds for rough pseudodifferential operators acting on Hardy spaces for Fourier integral operators (FIOs). These spaces were recently used to extend the optimal fixed-time $L^{p}$ regularity for wave equations, from equations with smooth coefficients to equations with rough coefficients. A key ingredient in the proof of these new regularity results involves bounds for rough pseudodifferential operators on the Hardy spaces for Fourier integral operators. The improved bounds for rough pseudodifferential operators in this article in turn lead to improvements in the $L^{p}$ regularity theory of wave equations with rough coefficients. Moreover, as in the case of paradifferential operators acting on classical function spaces, such bounds also imply algebraic properties of the Hardy spaces for Fourier integral operators.

\subsection{Setting}

The Hardy space $\HT^{1}_{FIO}(\Rn)$ for FIOs was introduced by Smith in \cite{Smith98a}. It is an invariant space for suitable FIOs, and it satisfies Sobolev embeddings that allow one to recover the optimal $L^{p}(\Rn)$ regularity of FIOs, obtained by Seeger, Sogge and Stein in \cite{SeSoSt91}. More precisely, an FIO $T$ of order zero, associated with a local canonical graph and having a compactly supported Schwartz kernel, satisfies $T:W^{s+s(p),p}(\Rn)\to W^{s-s(p),p}(\Rn)$ for all $1<p<\infty$ and $s\in\R$. Here and throughout, for $1\leq p\leq\infty$ we write 
\begin{equation}\label{eq:sp}
s(p):=\frac{n-1}{2}\Big|\frac{1}{2}-\frac{1}{p}\Big|.
\end{equation}
The construction of Smith was extended by Hassell, Portal and the author \cite{HaPoRo20} to a full range of invariant spaces for FIOs, denoted by $\Hp$ for $1\leq p\leq \infty$, satisfying the Sobolev embeddings
\begin{equation}\label{eq:Sobolev1}
W^{s(p),p}(\Rn)\subseteq\Hp\subseteq W^{-s(p),p}(\Rn)
\end{equation}
for $1<p<\infty$. By combining \eqref{eq:Sobolev1} with the invariance of $\Hp$ under FIOs, one indeed recovers the optimal $L^{p}(\Rn)$ regularity of these operators. However, the invariance of $\Hp$ under FIOs also allows for iterative constructions involving FIOs which are not possible when working directly on $L^{p}(\Rn)$, due to the loss of regularity which occurs in every iteration step. In turn, such iterative constructions are powerful tools for the study of wave equations with rough coefficients, due to another innovation by Smith which will be explained next.

In \cite{Smith98b}, Smith used techniques from paradifferential calculus, as introduced by Bony~\cite{Bony81} (see also Meyer \cite{Meyer81a,Meyer81b}), to construct a parametrix for wave equations with $C^{1,1}$ coefficients. More precisely, paradifferential calculus yields a decomposition of a differential operator $A$ with rough coefficients into a sum $A=A_{1}+A_{2}$ of pseudodifferential operators, where $A_{1}$ has smooth coefficients, and $A_{2}$ has rough coefficients but a lower differential order than $A$. Moreover, one can construct a parametrix for the smooth pseudodifferential equation $(\partial_{t}^{2}-A_{1})u(t)=0$, in terms of wave packet transforms and bicharacteristic flows. Given suitable mapping properties of $A_{2}$ on $L^{2}(\Rn)$, one can then use Duhamel's principle and an iterative correction procedure to obtain from this a parametrix for the full equation $(\partial_{t}^{2}-A)u(t)=0$ on $L^{2}(\Rn)$. The resulting parametrix was subsequently used by Smith, and by Tataru, to obtain powerful results for rough wave equations, such as Strichartz estimates \cite{Smith98b,Tataru00,Tataru01,Tataru02}, propagation of singularities \cite{Smith14}, the related spectral cluster estimates~\cite{Smith06}, and well-posedness of nonlinear wave equations \cite{Smith-Tataru05}.

Recently, Hassell and the author used the same procedure to extend the optimal fixed-time $L^{p}(\Rn)$ regularity for smooth wave equations to rough equations \cite{Hassell-Rozendaal20}. However, due to the loss of regularity that occurs when an FIO acts on $L^{p}(\Rn)$, the iterative correction procedure now has to take place on the Hardy spaces for FIOs. And, given the role of the operator $A_{2}$ in the construction above, this in turn means that one requires bounds for rough pseudodifferential operators on $\Hp$.

Whereas mapping properties of rough pseudodifferential operators on $L^{p}(\Rn)$ are classical \cite{Marschall88,Taylor91,Taylor00}, the first such properties on $\Hp$ were obtained by the author in \cite{Rozendaal20}. More precisely, it was shown there that, for a pseudodifferential operator $a(x,D)$ with symbol $a$ in the rough symbol class $C^{r}_{*}S^{0}_{1,1/2}$ (see Definition \ref{def:rough}), one has
\begin{equation}\label{eq:previous}
a(x,D):\Hp\to \Hp\quad\text{for all }1<p<\infty\text{ if }r>n-1.
\end{equation}
Here the Zygmund space $C^{r}_{*}(\Rn)$ from \eqref{eq:defZyg} measures the spatial regularity of the rough symbol, which apart from this roughness behaves like an $S^{0}_{1,1/2}$ symbol. We note that $C^{r}_{*}(\Rn)$ coincides with $C^{r}(\Rn)$ for $r\notin\N$ but strictly contains $C^{r-1,1}(\Rn)$ for $r\in\N$. 

In this article we will improve upon \eqref{eq:previous}.

\subsection{Main results}

The Sobolev space $\Hps\!=\!\lb D\rb^{-s}\Hp$ over $\Hp$ is introduced in Definition \ref{def:HpFIO} (see also \eqref{eq:Hpintro}). Here $\lb D\rb=(1+\sqrt{-\Delta})^{1/2}$. Recall the definition of $s(p)$ from \eqref{eq:sp}. A simplified version of our main result is as follows.

\begin{theorem}\label{thm:intro}
Let $r>0$, $p\in(1,\infty)$ and $a\in C^{r}_{*}S^{0}_{1,1/2}$. If $4s(p)<r$, then
\begin{equation}\label{eq:intromap}
a(x,D):\Hps\to \Hps
\end{equation}
continuously, for all $-r/2+s(p)<s<r-s(p)$.
\end{theorem}

Theorem \ref{thm:intro} follows from Theorem \ref{thm:main}, which also deals with the case where $4s(p)\geq r$. Moreover, \eqref{eq:intromap} also holds for $s=r-s(p)$ if $a$ has additional $\BMO$-type regularity. Under even more refined regularity assumptions on $a$, which are satisfied in the applications to rough wave equations in \cite{Hassell-Rozendaal20}, one may also let $s=-r/2+s(p)$.

Note that, since $4s(p)=2(n-1)|\frac{1}{2}-\frac{1}{p}|<n-1$ for any $1<p<\infty$, Theorem \ref{thm:intro} indeed improves upon \eqref{eq:previous}. Note also that, since the Zygmund space $C^{r}_{*}(\Rn)$ contains $C^{r}(\Rn)$, Theorem \ref{thm:intro} applies in particular to pseudodifferential operators whose symbols have $C^{r}$ regularity (or $C^{r-1,1}$ regularity if $r\in\N$). 

By combining Theorem \ref{thm:intro} with the parametrix approach described above, one obtains improved results on the fixed-time $L^{p}(\Rn)$ regularity of wave equations with rough coefficients. More precisely, in a first version of \cite{Hassell-Rozendaal20}, which relied on \cite{Rozendaal20}, the optimal $L^{p}(\Rn)$ regularity for wave equations with smooth coefficients was extended to equations with $C^{1,1}(\Rn)$ coefficients for $n\leq 2$. On the other hand, at least $C^{2+\veps}(\Rn)$ regularity of the coefficients was required for $n\geq3$. This restriction arose from the mapping property in \eqref{eq:previous}, and not from other parts of the parametrix construction (which require $C^{1,1}(\Rn)$ regularity of the coefficients). By contrast, Theorem \ref{thm:main} yields the optimal $L^{p}(\Rn)$ regularity for wave equations with $C^{1,1}(\Rn)$ coefficients in all dimensions, as long as $2s(p)<1$. These improved results are contained in a revised version of \cite{Hassell-Rozendaal20}. 

However, the results in this article are also of independent interest. For example, they yield algebraic properties of the Hardy spaces for FIOs. To understand why the algebraic structure of the Hardy spaces for FIOs is nontrivial, note that the $\Hp$ norm is given by (see Definition \ref{def:HpFIO})
\begin{equation}\label{eq:Hpintro}
\|f\|_{\Hp}=\|q(D)f\|_{L^{p}(\Rn)}+\Big(\int_{S^{n-1}}\|\ph_{\w}(D)f\|_{L^{p}(\Rn)}^{p}\ud\w\Big)^{1/p}
\end{equation}
for $1<p<\infty$ and $f\in\Hp$. Here $q\in C^{\infty}_{c}(\Rn)$, and the Fourier multiplier $\ph_{\w}(D)$ localizes in frequency to a paraboloid in the direction of $\w\in S^{n-1}$. Since Fourier multipliers do not commute with multiplication operators, it is unclear which multiplication operators preserve $\Hp$. Moreover, the sharpness of the Sobolev embeddings in \eqref{eq:Sobolev1} suggests that some smoothness of the multiplication operator might be necessary. By combining Theorem \ref{thm:main} with a paradifferential smoothing procedure, it is shown in Corollary \ref{cor:alldelta} that multiplication with an element of $C^{r}_{*}(\Rn)$ preserves $\Hps$, as long as $2s(p)<r$ and $-r/2+s(p)<s<r-s(p)$ (see Remark \ref{rem:mult}).

Another motivation for this article is the development of a paradifferential calculus on the Hardy spaces for FIOs. Given the powerful applications of paradifferential calculus to propagation of singularities and nonlinear partial differential equations, and because of the new tools that the Hardy spaces for FIOs provide, we expect that such a paradifferential calculus will also have other applications to the $L^{p}$ theory of wave equations.

\subsection{Overview of the proof}

To deduce Theorem \ref{thm:intro} we do not improve upon the methods that yielded \eqref{eq:previous} per se. In \cite{Rozendaal20}, \eqref{eq:previous} was proved by adapting classical paradifferential methods to the dyadic-parabolic, or second dyadic, decomposition of phase space which has been a crucial tool in the $L^{p}$ theory of FIOs since its inception \cite{Fefferman73b}. This decomposition is in turn embedded into the Hardy spaces for FIOs, through the combination of the parabolic frequency localizations $\ph_{\w}(D)$ in \eqref{eq:Hpintro}, and the dyadic frequency localizations in Littlewood--Paley theory for $L^{p}(\Rn)$. To adapt paradifferential calculus to this decomposition and prove \eqref{eq:previous}, one first uses a symbol decomposition from \cite{Coifman-Meyer78} to reduce to symbols with a simpler structure. Then one adapts the standard paraproduct paradigm, which involves grouping together the frequencies of functions into dyadic annuli and keeping track of their interaction when two functions are multiplied, to the dyadic-parabolic decomposition. Finally, an anisotropic Mikhlin multiplier theorem allows one to estimate away the parabolic $\ph_{\w}(D)$, after which dyadic Littlewood--Paley theory applies to $L^{p}(\Rn)$. However, it is not clear how to modify these methods, which apply to all $1<p<\infty$ simultaneously as long as $r>n-1$, to obtain boundedness on $\Hp$ for a fixed $p\neq 2$ if $r\leq n-1$.

Instead, we use \eqref{eq:previous} as a black box, and interpolate it with the classical result that a pseudodifferential operator with a $C^{r}_{*}S^{0}_{1,1/2}$ symbol is bounded on $L^{2}(\Rn)$ for any $r>0$. Since the Hardy spaces for FIOs form a complex interpolation scale, and because $\HT^{2}_{FIO}(\Rn)=L^{2}(\Rn)$, such an interpolation procedure is possible. However, we cannot rely on the same type of interpolation that has classically been used in Calder\'{o}n-Zygmund theory or, for that matter, in the proof of the optimal $L^{p}(\Rn)$ regularity of FIOs. In the latter case an FIO of order zero loses $(n-1)/2$ derivatives on the classical local Hardy space, and an FIO of order zero is bounded on $L^{2}(\Rn)$. Then interpolation of analytic families of operators, and duality, yields the optimal $L^{p}(\Rn)$ regularity of FIOs for all $1<p<\infty$. 

Such an approach does not work in our setting, since we want to show that a rough pseudodifferential operator $a(x,D)$ of order zero is bounded on $\Hp$ for a restricted range of $p$, and we cannot afford to lose any derivatives. Instead, we interpolate the regularity of the symbol $a\in C^{r}_{*}S^{0}_{1,1/2}$. More precisely, for $\kappa,\lambda\in\R$ with $\kappa+\lambda<r$ and $z\in\C$ with $0\leq \Real(z)\leq 1$, one can apply the operator $(1+\Delta)^{(\kappa z+\lambda)/2}$ to $a$ to construct an $a_{z}\in C^{r-\Real(\kappa z+\lambda)}_{*}S^{0}_{1,1/2}$ satisfying uniform symbol estimates in $\Imag(z)$. Moreover, for any $\delta>0$, one can choose $\kappa$ and $\lambda$ such that $a_{it}\in C^{n-1+\delta}_{*}S^{0}_{1,1/2}$ and $a_{1+it}\in C^{\delta}_{*}S^{0}_{1,1/2}$ for all $t\in\R$. Then $a_{1+it}(x,D)$ is bounded on $L^{2}(\Rn)$, and $a_{it}(x,D)$ is bounded on $\HT^{1+\delta}_{FIO}(\Rn)$, by \eqref{eq:previous}. As long as $4s(p)<r$, one can additionally find a $\theta\in[0,1]$ such that $\frac{1}{p}=\frac{1-\theta}{1+\delta}+\frac{\theta}{2}$ and $a_{\theta}=a$. Then interpolation of analytic families of operators shows that $a(x,D)$ is bounded on $\Hp$, thereby proving Theorem \ref{thm:intro} for $s=0$. 

To prove Theorem \ref{thm:intro} for all $-r/2+s(p)<s<r-s(p)$, we need to modify this approach slightly. Indeed, the range of Sobolev indices $s$ for which one can expect boundedness of $a(x,D)$ on $\Hps$ is intrinsically tied to the regularity of $a$, as is already clear for multiplication operators acting on $\HT^{s,2}_{FIO}(\Rn)=W^{s,2}(\Rn)$. When directly applying the interpolation procedure above to $a$, one only obtains boundedness of $a_{1+it}(x,D)$ on $\HT^{s,2}_{FIO}(\Rn)$ for a small interval of $s$ around $0$. Then interpolation does not yield the full range of Sobolev exponents in Theorem \ref{thm:intro}. 

Instead, we first apply the same symbol smoothing procedure as in \cite{Rozendaal20} to $a$, to remove the low and high frequencies and deal with these separately. The low frequencies are not problematic, while the high frequencies can be dealt with using the Sobolev embeddings in \eqref{eq:Sobolev1} and classical results about rough pseudodifferential operators on $W^{s,p}(\Rn)$. This leads to the restriction on $s$ in Theorem \ref{thm:intro}. The remaining ``middle" frequencies, which have to be treated differently because of the dyadic-parabolic decomposition of the Hardy spaces for FIOs, were the most problematic in \cite{Rozendaal20}, but they do not lead to any restrictions on the range of Sobolev indices $s$. To prove Theorem \ref{thm:intro} one can thus apply the interpolation procedure above to the part of the symbol of $a$ which only contains these ``middle" frequencies.

\subsection{Organization of this article}

This article is organized as follows. In Section \ref{sec:Hardy} we collect some background on the Hardy spaces for Fourier integral operators, including a version for these spaces of the standard result on interpolation of analytic families of operators. In Section \ref{sec:symbols} we introduce our rough symbol classes, as well as the paradifferential symbol smoothing procedure which is used in the proof of Theorem \ref{thm:intro}. Section \ref{sec:prelim} in turn contains the various preliminary results which are mixed together in the proof of our main theorem. These include the aforementioned proposition on interpolation of rough symbols, as well as a more refined version of \eqref{eq:previous} and a lemma about pseudodifferential operators with rough symbols on $L^{p}(\Rn)$. Finally, Section \ref{sec:main} contains the proof of our main result, Theorem \ref{thm:main}, which in turn implies Theorem \ref{thm:intro}. In Corollary \ref{cor:alldelta} we deduce from Theorem \ref{thm:main} a more general statement about pseudodifferential operators with $C^{r}_{*}S^{0}_{1,\delta}$ symbols for $0\leq \delta\leq 1/2$ which applies in particular to multiplication operators.

\subsection*{Notation}\label{subsec:notation}

The natural numbers are $\N=\{1,2,\ldots\}$, and the nonnegative integers are $\Z_{+}=\N\cup\{0\}$. Throughout this article, we fix $n\in\N$ with $n\geq2$. Our techniques also apply for $n=1$, but in that case the results are classical, by Lemma \ref{lem:pseudoLp} and because $\HT^{s,p}_{FIO}(\R)=\HT^{s,p}(\R)$ for all $p\in[1,\infty]$ and $s\in\R$ (see \eqref{eq:classical} and \eqref{eq:sobolev}). 

For $\xi\in\Rn$ we write $\lb\xi\rb=(1+|\xi|^{2})^{1/2}$, and $\hat{\xi}=\xi/|\xi|$ if $\xi\neq 0$. We use multi-index notation, where $\partial^{\alpha}_{\xi}=\partial^{\alpha_{1}}_{\xi_{1}}\ldots\partial^{\alpha_{n}}_{\xi_{n}}$, and we set $\partial_{\xi}:=(\partial_{\xi_{1}},\ldots,\partial_{\xi_{n}})$.

The spaces of Schwartz functions and tempered distributions are $\Sw(\Rn)$ and $\Sw'(\Rn)$, respectively.  The Fourier transform of an $f\in\Sw'(\Rn)$ is denoted by $\F f$ or $\widehat{f}$. For $f\in\Ell^{1}(\Rn)$ and $\xi\in\Rn$ one has
\begin{align*}
\F f(\xi)=\int_{\Rn}e^{-i x\xi}f(x)\ud x.
\end{align*}
The Fourier multiplier with symbol $\ph\in\Sw'(\Rn)$ is denoted by $\ph(D)$. The space of bounded linear operators between Banach spaces $X$ and $Y$ is $\La(X,Y)$, and $\La(X):=\La(X,X)$. 

We write $f(s)\lesssim g(s)$ to indicate that $f(s)\leq Cg(s)$ for all $s$ and a constant $C>0$ independent of $s$, and similarly for $f(s)\gtrsim g(s)$ and $g(s)\eqsim f(s)$.

\section{Hardy spaces for Fourier integral operators}\label{sec:Hardy}

In this section we collect the prerequisite background on the Hardy spaces for Fourier integral operators.

\subsection{Definitions}

For notational convenience, throughout this article we write
\begin{equation}\label{eq:classical}
\HT^{s,p}(\Rn):=
\begin{cases}
W^{s,p}(\Rn)&\text{if }p\in(1,\infty),\\
\lb D\rb^{-s}\HT^{1}(\Rn)&\text{if }p=1,\\
\lb D\rb^{-s}\bmo(\Rn)&\text{if }p=\infty,
\end{cases}
\end{equation}
for $p\in[1,\infty]$ and $s\in\R$, with the natural associated norms. Here $\HT^{1}(\Rn)$ is the local Hardy space from \cite{Goldberg79}, and $\bmo(\Rn)$ is its dual. Recall that $\bmo(\Rn)$ consists of all $f\in\Sw'(\Rn)$ such that $q(D)f\in L^{\infty}(\Rn)$ and $(1-q)(D)f\in \BMO(\Rn)$, with
\[
\|f\|_{\bmo(\Rn)}=\|q(D)f\|_{L^{\infty}(\Rn)}+\|(1-q)(D)f\|_{\BMO(\Rn)}.
\]
Here and in the rest of this article, $q\in C^{\infty}_{c}(\Rn)$ is such that $q(\xi)=1$ for $|\xi|\leq 2$.

Fix a non-negative radial $\ph\in C^{\infty}_{c}(\Rn)$ such that $\ph(\xi)=0$ for $|\xi|>1$, and $\ph\equiv1$ in a neighborhood of zero. Later on we will require the support of $\ph$ to be sufficiently small. For $\w\in S^{n-1}$, $\sigma>0$ and $\xi\in\Rn\setminus\{0\}$, we write $\ph_{\w,\sigma}(\xi):=c_{\sigma}\ph(\tfrac{\hat{\xi}-\w}{\sqrt{\sigma}})$, where $c_{\sigma}:=(\int_{S^{n-1}}\ph(\frac{e_{1}-\nu}{\sqrt{\sigma}})^{2}\ud\nu)^{-1/2}$ for $e_{1}=(1,0,\ldots,0)$ the first basis vector of $\Rn$ (this choice is immaterial). Also set $\ph_{\w,\sigma}(0):=0$. Let $\Psi\in C^{\infty}_{c}(\Rn)$ be non-negative, radial, such that $\Psi(\xi)=0$ if $|\xi|\notin[1/2,2]$, and such that $\int_{0}^{\infty}\Psi(\sigma\xi)^{2}\frac{\ud \sigma}{\sigma}=1$ for all $\xi\neq 0$. Next, for $\w\in S^{n-1}$ and $\xi\in\Rn$, set 
\[
\ph_{\w}(\xi):=\int_{0}^{4}\Psi(\sigma\xi)\ph_{\w,\sigma}(\xi)\frac{\ud\sigma}{\sigma}.
\]
Some properties of these functions are as follows (see \cite[Remark 3.3]{Rozendaal21}):
\begin{enumerate}
\item For each $\w\in S^{n-1}$, the function $\ph_{\w}$ is supported on a paraboloid in the direction of $\w$. More precisely, for $\xi\neq0$, one has 
\[
\ph_{\w}(\xi)=0\text{ if }|\xi|<\tfrac{1}{8}\text{ or }|\hat{\xi}-\w|>2|\xi|^{-1/2}.
\]
\item For all $\alpha\in\Z_{+}^{n}$ and $\beta\in\Z_{+}$ there exists a $C_{\alpha,\beta}\geq0$ such that
\[
|(\w\cdot \partial_{\xi})^{\beta}\partial^{\alpha}_{\xi}\ph_{\w}(\xi)|\leq C_{\alpha,\beta}|\xi|^{\frac{n-1}{4}-\frac{|\alpha|}{2}-\beta}
\]
for all $\w\in S^{n-1}$ and $\xi\neq0$. In particular, the inverse Fourier transforms of the functions $\xi\mapsto \lb\xi\rb^{-2n}\ph_{\w}(\xi)$ are uniformly bounded in $L^{1}(\Rn)$, and
\begin{equation}\label{eq:Lqmap2}
\sup_{\w\in S^{n-1}}\|\lb D\rb^{-2n}\ph_{\w}(D)\|_{\La(L^{p}(\Rn))}<\infty
\end{equation}
for all $p\in[1,\infty]$.
\end{enumerate}

We can now define the Hardy spaces for Fourier integral operators.

\begin{definition}\label{def:HpFIO}
For $p\in[1,\infty)$, $\Hp$ consists of those $f\in\Sw'(\Rn)$ such that $q(D)f\in L^{p}(\Rn)$, $\ph_{\w}(D)f\in \HT^{p}(\Rn)$ for almost all $\w\in S^{n-1}$, and
\[
\Big(\int_{S^{n-1}}\|\ph_{\w}(D)f\|_{\HT^{p}(\Rn)}^{p}\ud\w\Big)^{1/p}<\infty,
\]
endowed with the norm
\[
\|f\|_{\Hp}:=\|q(D)f\|_{L^{p}(\Rn)}+\Big(\int_{S^{n-1}}\|\ph_{\w}(D)f\|_{\HT^{p}(\Rn)}^{p}\ud\w\Big)^{1/p}.
\]
Moreover, $\HT^{\infty}_{FIO}(\Rn):=(\HT^{1}_{FIO}(\Rn))^{*}$. For $p\in[1,\infty]$ and $s\in\R$, $\Hps$ consists of all $f\in\Sw'(\Rn)$ such that $\lb D\rb^{s}f\in\Hp$, endowed with the norm
\[
\|f\|_{\Hps}:=\|\lb D\rb^{s}f\|_{\Hp}.
\] 
\end{definition}

It is straightforward to see that, for $1\leq p<\infty$ and $s\in\R$, one has
\begin{equation}\label{eq:HpFIOnorm}
\|f\|_{\Hps}\eqsim \|q(D)f\|_{L^{p}(\Rn)}+\Big(\int_{S^{n-1}}\|\ph_{\w}(D)f\|_{\HT^{s,p}(\Rn)}^{p}\ud\w\Big)^{1/p}
\end{equation}
for all $f\in\Hps$, with implicit constants independent of $f$. 

We note that the present definition of $\Hp$ is not how these spaces were originally defined. In \cite{Smith98a} and \cite{HaPoRo20} they were defined in terms of conical square functions, or equivalently tent spaces, over the cosphere bundle. That $\Hp$ can be equivalently described in this manner was shown by the author \cite{Rozendaal21} for $1<p<\infty$, and by Fan, Liu, Song and the author \cite{FaLiRoSo19} for $p=1$.

\subsection{Properties}

We now collect some properties of the Hardy spaces for Fourier integral operators.

Firstly, as was already indicated in \eqref{eq:Sobolev1} in the case where $p\in(1,\infty)$ and $s=0$, the following Sobolev embeddings hold for all $p\in[1,\infty]$ and $s\in\R$: 
\begin{equation}\label{eq:sobolev}
\HT^{s+s(p),p}(\Rn)\subseteq\Hps\subseteq\HT^{s-s(p),p}(\Rn).
\end{equation}
Here $s(p)$ is as in \eqref{eq:sp}, and the exponents in these embeddings cannot be improved. In particular, one has $\HT^{s,2}_{FIO}(\Rn)=W^{s,2}(\Rn)$.

Secondly, the Hardy spaces for Fourier integral operators form a complex interpolation scale. More precisely, let $p_{0},p_{1}\in[1,\infty]$, $s_{0},s_{1}\in\R$ and $\theta\in[0,1]$ be given, and let $p\in[1,\infty]$ and $s\in\R$ be such that $\frac{1}{p}=\frac{1-\theta}{p_{0}}+\frac{\theta}{p_{1}}$ and $s=(1-\theta)s_{0}+\theta s_{1}$. Then 
\begin{equation}\label{eq:HpFIOinter}
[\HT^{s_{0},p_{0}}_{FIO}(\Rn),\HT^{s_{1},p_{1}}_{FIO}(\Rn)]_{\theta}=\Hps.
\end{equation}
This statement extends \cite[Proposition 6.7]{HaPoRo20}, and it follows in the same manner as that proposition, from established results about complex interpolation of (weighted) tent spaces. For a proof see \cite[Corollary 3.5]{Hassell-Rozendaal20}.

We also note, for later use, that the Schwartz functions with compact Fourier support lie dense in $\Hps$ for all $p\in[1,\infty)$ and $s\in\R$. For $s=0$, the fact that the Schwartz functions lie dense is \cite[Proposition 6.6]{HaPoRo20}, from which the corresponding statement for general $s$ readily follows. To see that one may assume that the functions have compact Fourier support, one can either approximate general Schwartz functions, or inspect the proof of \cite[Proposition 6.6]{HaPoRo20} and use results about density of compactly supported functions in tent spaces.

An essential role will be played in this article by the following version of the standard result on interpolation of analytic families of operators in our setting. Throughout, set
\begin{equation}\label{eq:S}
S:=\{z\in\C\mid 0<\Real(z)<1\}.
\end{equation}
Recall that, for Banach spaces $X$ and $Y$ embedded into a Hausdorff topological vector space $Z$, the subspaces $X\cap Y$ and $X+Y$ are Banach spaces with the norms
\[
\|z\|_{X\cap Y}:=\max(\|z\|_{X},\|z\|_{Y})
\]
for $z\in X\cap Y$, and
\[
\|z\|_{X+Y}:=\inf\{\|x\|_{X}+\|y\|_{Y}\mid x\in X,y\in Y, x+y=z\}
\]
for $z\in X+Y$.

\begin{proposition}\label{prop:interan}
Let $p_{0},p_{1}\in[1,\infty]$, $s_{0},s_{1},t_{0},t_{1}\in\R$ and $\theta\in[0,1]$. Let $p\in[1,\infty]$ and $s,t\in\R$ be such that $\frac{1}{p}=\frac{1-\theta}{p_{0}}+\frac{\theta}{p_{1}}$, $s=(1-\theta)s_{0}+\theta s_{1}$ and $t=(1-\theta)t_{0}+\theta t_{1}$. Let  
\[
F:\overline{S}\to\La\big(\HT^{s_{0},p_{0}}_{FIO}(\Rn)\cap \HT^{s_{1},p_{1}}_{FIO}(\Rn),\HT^{t_{0},p_{0}}_{FIO}(\Rn)+\HT^{t_{1},p_{1}}_{FIO}(\Rn)\big).
\]
Suppose that the following conditions hold for all $f\in \HT^{s_{0},p_{0}}_{FIO}(\Rn)\cap \HT^{s_{1},p_{1}}_{FIO}(\Rn)$.
\begin{enumerate}
\item\label{it:inter1} The $\HT^{t_{0},p_{0}}_{FIO}(\Rn)+\HT^{t_{1},p_{1}}_{FIO}(\Rn)$-valued map $z\mapsto F(z)f$ is continuous and bounded on $\overline{S}$, and holomorphic on $S$.
\item\label{it:inter2} One has 
\[
[\tau\mapsto F(i\tau)f]\in C(\R;\HT^{t_{0},p_{0}}_{FIO}(\Rn))\text{ and }[\tau\mapsto F(1+i\tau)f]\in C(\R;\HT^{t_{1},p_{1}}_{FIO}(\Rn)).
\]
\item\label{it:inter3} One has $F(i\tau):\HT^{s_{0},p_{0}}_{FIO}(\Rn)\to \HT^{t_{0},p_{0}}_{FIO}(\Rn)$ and $F(1+i\tau):\HT^{s_{1},p_{1}}_{FIO}(\Rn)\to \HT^{t_{1},p_{1}}_{FIO}(\Rn)$ for each $\tau\in\R$, with
\[
M_{0}:=\sup_{\tau\in\R}\|F(i\tau)\|_{\La(\HT^{s_{0},p_{0}}_{FIO}(\Rn),\HT^{t_{0},p_{0}}_{FIO}(\Rn))}<\infty
\]
and
\[
M_{1}:=\sup_{\tau\in\R}\|F(1+i\tau)\|_{\La(\HT^{s_{1},p_{1}}_{FIO}(\Rn),\HT^{t_{1},p_{1}}_{FIO}(\Rn))}<\infty.
\]
\end{enumerate}
Then $F(\theta):\Hps\to \Hpt$ is bounded, with 
\[
\|F(\theta)\|_{\La(\Hps,\Hpt))}\leq M_{0}^{1-\theta}M_{1}^{\theta}.
\]
\end{proposition}
\begin{proof}
This is just a combination of a Banach space version of the proposition (see e.g.~\cite[Theorem 2.1.7]{Lunardi09}) and \eqref{eq:HpFIOinter}.
\end{proof}

We note that if the operator norms of the $F(z)$ are uniformly bounded, i.e.~if there exists an $M\geq0$ such that
\[
\|F(z)\|_{\La(\HT^{s_{0},p_{0}}_{FIO}(\Rn)\cap \HT^{s_{1},p_{1}}_{FIO}(\Rn),\HT^{t_{0},p_{0}}_{FIO}(\Rn)+\HT^{t_{1},p_{1}}_{FIO}(\Rn))}\leq M
\]
for all $z\in\overline{S}$, then it suffices to check the continuity and analyticity conditions in \eqref{it:inter1} and \eqref{it:inter2} for all $f$ in a dense subset of $\HT^{s_{0},p_{0}}_{FIO}(\Rn)\cap \HT^{s_{1},p_{1}}_{FIO}(\Rn)$.

\section{Rough symbols}\label{sec:symbols}

In this section we introduce our classes of rough symbols, as well as a paradifferential symbol smoothing procedure. 

\subsection{Function spaces}

To define our rough symbol classes, we first introduce function spaces which measure the regularity of a symbol in its spatial variable.

Throughout, we fix a Littlewood--Paley decomposition $(\psi_{j})_{j=0}^{\infty}\subseteq C^{\infty}_{c}(\Rn)$. That is, for all $\xi\in\Rn$ one has
\[
\sum_{j=0}^{\infty}\psi_{j}(\xi)=1,
\]
$\psi_{0}(\xi)=0$ if $|\xi|>1$, $\psi_{1}(\xi)=0$ if $|\xi|\notin [1/2,2]$, and $\psi_{j}(\xi)=\psi_{1}(2^{-j+1}\xi)$ for all $j>1$. For notational convenience we also write $\psi_{j}:=0$ for $j<0$.

For $r\in\R$, we let the \emph{Zygmund space} $C^{r}_{*}(\Rn)$ consist of those $f\in\Sw'(\Rn)$ such that $\psi_{j}(D)f\in L^{\infty}(\Rn)$ for all $j\geq0$, with
\begin{equation}\label{eq:defZyg}
\|f\|_{C^{r}_{*}(\Rn)}:=\sup_{j\geq0}2^{jr}\|\psi_{j}(D)f\|_{L^{\infty}(\Rn)}<\infty.
\end{equation}
Clearly the contractive embedding
\begin{equation}\label{eq:Zygmund1}
C^{r}_{*}(\Rn)\subseteq C^{t}_{*}(\Rn)
\end{equation}
holds for all $t<r$. 

Let $r>0$ and write $r=l+s$ with $l\in\Z_{+}$ and $s\in(0,1]$. Then (see~\cite{Triebel10})
\begin{equation}\label{eq:Zygmund2}
\Hrinf(\Rn)\subsetneq C^{r}_{*}(\Rn)=C^{r}(\Rn)\cap L^{\infty}(\Rn)\subsetneq C^{l}(\Rn)\cap L^{\infty}(\Rn)
\end{equation}
if $r\notin\N$, i.e.~if $s\in(0,1)$, and
\[
C^{r}(\Rn)\cap L^{\infty}(\Rn)\subsetneq C^{l,1}(\Rn)\cap L^{\infty}(\Rn)\subsetneq \Hrinf(\Rn)\subsetneq C^{r}_{*}(\Rn)
\]
if $r\in\N$, i.e.~if $s=1$. Here $\Hrinf(\Rn)$ is as in \eqref{eq:classical}. We recall that, for $r\notin\N$, $C^{r}(\Rn)$ consists of those $f\in C^{l}(\Rn)$ such that for each $\alpha\in\Z_{+}^{n}$ with $|\alpha|=l$, the partial derivative $\partial_{x}^{\alpha}f$ is H\"{o}lder continuous with parameter $s$. Also, $C^{l,1}(\Rn)$ consists of those $f\in C^{l}(\Rn)$ such that $\partial_{x}^{\alpha}f$ is Lipschitz for each $\alpha\in\Z_{+}^{n}$ with $|\alpha|=l$. In particular, by combining \eqref{eq:Zygmund1} and \eqref{eq:Zygmund2}, it follows that there exists a constant $M_{r}\geq0$ such that $C^{r}_{*}(\Rn)\subseteq C^{l}(\Rn)\cap L^{\infty}(\Rn)$ and, for all $f\in C^{r}_{*}(\Rn)$,
\begin{equation}\label{eq:Zygmund3}
\max_{|\alpha|\leq l}\|\partial_{x}^{\alpha}f\|_{L^{\infty}(\Rn)}\leq M_{r}\|f\|_{C^{r}_{*}(\Rn)}.
\end{equation}
We also note that $C^{r}_{*}(\Rn)$ is equal to the Besov space $B^{r}_{\infty,\infty}(\Rn)$.

\subsection{Rough symbols}

Recall that, for $m\in\R$ and $\delta\in[0,1]$, the symbol class $S^{m}_{1,\delta}$ is the space of $a\in C^{\infty}(\R^{2n})$ such that, for all $\alpha,\beta\in\Z_{+}^{n}$, there exists an $M_{\alpha,\beta}\geq0$ with
\[
|\partial_{x}^{\beta}\partial_{\eta}^{\alpha}a(x,\eta)|\leq M_{\alpha,\beta}\lb\eta\rb^{m-|\alpha|+\delta|\beta|}
\]
for all $x,\eta\in\Rn$. We now introduce versions of these symbols that have limited regularity in the $x$ variable.

\begin{definition}\label{def:rough}
Let $r>0$, $m\in\R$, $\delta\in[0,1]$ and $l\in\N$. Then $C^{r}_{*}S^{m,l}_{1,\delta}$ is the collection of $a:\R^{2n}\to\C$ for which there exists an $M\geq0$ such that, for each $\alpha\in\Z_{+}^{n}$ with $|\alpha|\leq l$, the following properties hold:
\begin{enumerate}
\item\label{it:symbol1} For all $x,\eta\in\Rn$ one has $a(x,\cdot)\in C^{l}(\Rn)$ and
\begin{equation}\label{eq:etareg}
|\partial_{\eta}^{\alpha}a(x,\eta)|\leq M\lb\eta\rb^{m-|\alpha|}.
\end{equation}
\item\label{it:symbol2} For all $\eta\in\Rn$ one has $\partial_{\eta}^{\alpha}a(\cdot,\eta)\in C^{r}_{*}(\Rn)$ and
\begin{equation}\label{eq:xreg}
\|\partial_{\eta}^{\alpha}a(\cdot,\eta)\|_{C^{r}_{*}(\Rn)}\leq M\lb\eta\rb^{m-|\alpha|+r\delta}.
\end{equation}
\end{enumerate} 
Moreover, $C^{r}_{*}S^{m}_{1,\delta}:=\cap_{l\geq1}C^{r}_{*}S^{m,l}_{1,\delta}$, and $\Hrinf S^{m}_{1,\delta}$ consists of all $a\in C^{r}_{*}S^{m}_{1,\delta}$ with the following additional property. For each $\alpha\in\Z_{+}^{n}$ there exists an $M_{\alpha}\geq0$ such that, for all $\eta\in\Rn$, one has $\partial_{\eta}^{\alpha}a(\cdot,\eta)\in \Hrinf(\Rn)$ and 
\[
\|\partial_{\eta}^{\alpha}a(\cdot,\eta)\|_{\Hrinf(\Rn)}\leq M_{\alpha}\lb\eta\rb^{m-|\alpha|+r\delta}.
\]
\end{definition}

Note that $C^{r}_{*}S^{m,l}_{1,\delta}$ is a Banach space, with norm $\|a\|_{C^{r}_{*}S^{m,l}_{1,\delta}}$ given by the smallest $M\geq0$ such that \eqref{eq:etareg} and \eqref{eq:xreg} hold. Moreover, $C^{r}_{*}S^{m}_{1,\delta}$ is a locally convex space with the topology generated by the associated collection of seminorms, and similarly for $\Hrinf S^{m}_{1,\delta}$.

\begin{remark}\label{rem:finitel}
The reason for including the additional parameter $l\in\N$ in Definition \ref{def:rough} is that we will need to keep track of estimates for operator norms of $a(x,D)$. 
\end{remark}

Note that
\[
S^{m}_{1,\delta}\subsetneq \Hrinf S^{m,l}_{1,\delta}\subsetneq C^{r}_{*}S^{m,l}_{1,\delta}
\]
for all $r>0$, $m\in\R$, $\delta\in[0,1]$ and $l\in\N$. Moreover, the contractive embeddings  
\begin{equation}\label{eq:symbolinc}
C^{r}_{*}S^{m-(\beta-\delta)r,l}_{1,\beta}\subseteq C^{r}_{*}S^{m,l}_{1,\delta}
\end{equation}
and
\[
\Hrinf S^{m-(\beta-\delta)r}_{1,\beta}\subseteq \Hrinf S^{m}_{1,\delta}
\]
hold for all $\beta\in[\delta,1]$ and $l\in\N$, as one can readily check.

Given $a\in C^{r}_{*}S^{m,l}_{1,\delta}$ for some $r>0$, $m\in\R$, $\delta\in[0,1]$ and $l\in\N$, the pseudodifferential operator $a(x,D):\Sw(\Rn)\to\Sw'(\Rn)$ with symbol $a$ is given by
\begin{equation}\label{eq:pseudodef}
a(x,D)f(x):=\frac{1}{(2\pi)^{n}}\int_{\Rn}e^{ix\eta}a(x,\eta)\wh{f}(\eta)\ud\eta
\end{equation}
for $f\in\Sw(\Rn)$ and $x\in\Rn$.

\begin{remark}\label{rem:dualitydef}
In \eqref{eq:pseudodef} we only defined $a(x,D)f$ for $f\in\Sw(\Rn)$, whereas the Schwartz functions do not lie dense in $\HT^{s,\infty}_{FIO}(\Rn)$ for any $s\in\R$. Since $\HT^{s,\infty}_{FIO}(\Rn)$ will hardly play a role in this article, we will not concern ourselves with the subtleties of the adjoints of pseudodifferential operators with rough symbols (see, however, \cite[Remark 2.6]{Rozendaal20} for slightly more on this).
\end{remark}

\subsection{Symbol smoothing}

Next, we introduce a symbol smoothing procedure which decomposes a rough symbol as a sum of a smooth part and a rough part with lower differential order. Let $a\in C^{r}_{*}S^{m}_{1,\delta}$ for $r>0$, $m\in\R$ and $\delta\in[0,1]$. Let $\beta\in[\delta,1]$ be given, and recall from Section \ref{sec:Hardy} that $\ph\in C^{\infty}_{c}(\Rn)$ is such that $\ph\equiv1$ near zero. Now set, for $x,\eta\in\Rn$,
\[
a^{\sharp}_{\beta}(x,\eta):=\sum_{k=0}^{\infty}\big(\ph(2^{-\beta k}D)a(\cdot,\eta)\big)(x)\psi_{k}(\eta)
\]
and
\[
a^{\flat}_{\beta}(x,\eta):=a(x,\eta)-a^{\sharp}_{\beta}(x,\eta)=\sum_{k=0}^{\infty}\big((1-\ph)(2^{-\beta k}D)a(\cdot,\eta)\big)(x)\psi_{k}(\eta).
\]
For the final identity we used that $\sum_{k=0}^{\infty}\psi_{k}(\eta)=1$. The decomposition $a=a^{\sharp}_{\beta}+a^{\flat}_{\beta}$ has the following properties, cf.~\cite[Lemma 2.8]{Rozendaal20} (see also \cite[Section 1.3]{Taylor91}).

\begin{lemma}\label{lem:smoothing}
Let $r>0$, $m\in\R$ and $\delta,\beta\in[0,1]$ with $\beta\geq\delta$. Then, for each $a\in C^{r}_{*}S^{m}_{1,\delta}$, one has $a^{\sharp}_{\beta}\in S^{m}_{1,\beta}$ and $a^{\flat}_{\beta}\in C^{r}_{*}S^{m-(\beta-\delta)r}_{1,\beta}$. 
If $a\in \Hrinf S^{m}_{1,\delta}$, then one additionally has $a^{\flat}_{\beta}\in \Hrinf S^{m-(\beta-\delta)r}_{1,\beta}$. 
\end{lemma}

\section{Preliminary results}\label{sec:prelim}

This section contains several preliminary results which will play an important role in the proof of our main theorem. We first prove a technical statement about rough symbols which will be used for the interpolation argument in the proof of our main theorem, and then we collect some results on boundedness of pseudodifferential operators.

\subsection{Results for interpolation}

In this subsection we prove a technical proposition about rough symbols, for the interpolation procedure in the proof of our main result.

We first collect two basic statements about Zygmund spaces. Recall that $S=\{z\in\C\mid 0<\Real(z)<1\}$, as in \eqref{eq:S}.

\begin{lemma}\label{lem:Zyg}
There exists an $M\geq 0$ such that 
\begin{equation}\label{eq:Zyg3}
\begin{aligned}
\tfrac{1}{3}2^{(j-1)\Real(\kappa z+\lambda)}\|\psi_{j}(D)f\|_{L^{\infty}(\Rn)}&\leq \|\psi_{j}(D)f\|_{C^{\Real(\kappa z+\lambda)}_{*}(\Rn)}\\
&\leq M2^{(j+1)\Real(\kappa z+\lambda)}\|\psi_{j}(D)f\|_{L^{\infty}(\Rn)}
\end{aligned}
\end{equation}
for all $\kappa,\lambda\in \R$, $z\in\overline{S}$, $j\geq0$ and $f\in L^{\infty}(\Rn)$. Moreover, for all $\kappa,\lambda\in\R$ there exists an $M_{\kappa,\lambda}\geq0$ such that, for all $z\in\overline{S}$, one has
\begin{equation}\label{eq:Zyg1}
\|e^{(\kappa z+\lambda)^{2}}\lb D\rb^{\kappa z+\lambda}\|_{\La(C^{0}_{*}(\Rn),C^{-\Real(\kappa z+\lambda)}_{*}(\Rn))}\leq M_{\kappa,\lambda}
\end{equation}
and
\begin{equation}\label{eq:Zyg2}
\|e^{(\kappa z+\lambda)^{2}}\lb D\rb^{\kappa z+\lambda}\|_{\La(C^{\Real(\kappa z+\lambda)}_{*}(\Rn),C^{0}_{*}(\Rn))}\leq M_{\kappa,\lambda}.
\end{equation}
\end{lemma}
\begin{proof}
For the first statement, use that $\psi_{j}=\sum_{i=j-1}^{j+1}\psi_{i}\psi_{j}$ to write
\begin{align*}
\tfrac{1}{3}2^{(j-1)\Real(\kappa z+\lambda)}\|\psi_{j}(D)f\|_{L^{\infty}(\Rn)}&=\tfrac{1}{3}2^{(j-1)\Real(\kappa z+\lambda)}\Big\|\sum_{i=j-1}^{j+1}\psi_{i}(D)\psi_{j}(D)f\Big\|_{L^{\infty}(\Rn)}\\
&\leq \tfrac{1}{3}\sum_{i=j-1}^{j+1}2^{i\Real(\kappa z+\lambda)}\|\psi_{i}(D)\psi_{j}(D)f\|_{L^{\infty}(\Rn)}\\
&\leq \|\psi_{j}(D)f\|_{C^{\Real(\kappa z+\lambda)}_{*}(\Rn)}\\
&=\sup_{i\geq0}2^{i\Real(\kappa z+\lambda)}\|\psi_{i}(D)\psi_{j}(D)f\|_{L^{\infty}(\Rn)}\\
&\leq M2^{(j+1)\Real(\kappa z+\lambda)}\|\psi_{j}(D)f\|_{L^{\infty}(\Rn)}
\end{align*}
for all $z\in\overline{S}$, $j\geq0$ and $f\in L^{\infty}(\Rn)$, with $M:=\sup_{i\geq0}\|\F^{-1}\psi_{i}\|_{L^{1}(\Rn)}$. In the final inequality we used that $\psi_{i}\psi_{j}=0$ for $i\notin\{j-1,j,j+1\}$.

The remaining two statements follow from Fourier multiplier theory. Indeed, by cutting off the frequencies of $f$, for both \eqref{eq:Zyg1} and \eqref{eq:Zyg2} it suffices to show that
\begin{equation}\label{eq:Zygbound}
e^{-\Imag(z)^{2}}\|2^{-j(\kappa z+\lambda)}\lb D\rb^{\kappa z+\lambda}\psi_{j}(D)f\|_{L^{\infty}(\Rn)}\lesssim \|\psi_{j}(D)f\|_{L^{\infty}(\Rn)}
\end{equation}
for an implicit constant independent of $z\in\overline{S}$, $j\geq0$ and $f\in C^{0}_{*}(\Rn)$. To this end, let $\wt{\psi}\in\Sw(\Rn)$ be such that $\wt{\psi}\equiv 1$ on $\supp(\psi)$, and $\wt{\psi}(\xi)=0$ if $|\xi|\notin[1/4,4]$. Set $\wt{\psi}_{j,z}(\xi):=2^{-j(\kappa z+\lambda)}\lb\xi\rb^{\kappa z+\lambda}\wt{\psi}(2^{-j+1}\xi)$ for $j\geq1$ and $\xi\in\Rn$. Then
\[
2^{-j(\kappa z+\lambda)}\lb D\rb^{\kappa z+\lambda}\psi_{j}(D)f=\wt{\psi}_{j,z}(D)\psi_{j}(D)f.
\]
Now, for each $N\geq0$ there exist $C,L\geq0$ such that, for all $\alpha\in\Z_{+}^{n}$ with $|\alpha|\leq N$, and all $j\geq0$ and $z\in\overline{S}$, one has
\[
\sup_{\xi\in\Rn}\lb\xi\rb^{|\alpha|}\big|\partial_{\xi}^{\alpha}\wt{\psi}_{j,z}(\xi)\big|\leq C(1+|\Imag(z)|)^{L}.
\]
Hence the Mikhlin multiplier theorem on the Besov space $C^{0}_{*}(\Rn)=B^{0}_{\infty,\infty}(\Rn)$ (see \cite[Section 2.3.7]{Triebel10}) shows that the collection $\{e^{-\Imag(z)^{2}}\wt{\psi}_{j,z}(D)\mid j\geq0, z\in\overline{S}\}$ is uniformly bounded in $\La(C^{0}_{*}(\Rn))$. Thus we can use \eqref{eq:Zyg3} to write
\begin{align*}
e^{-\Imag(z)^{2}}\|2^{-j(\kappa z+\lambda)}\lb D\rb^{\kappa z+\lambda}\psi_{j}(D)f\|_{L^{\infty}(\Rn)}&\eqsim e^{-\Imag(z)^{2}}\|\wt{\psi}_{j}(D)\psi_{j}(D)f\|_{C^{0}_{*}(\Rn)}\\
&\lesssim \|\psi_{j}(D)f\|_{C^{0}_{*}(\Rn)}\eqsim\|\psi_{j}(D)f\|_{L^{\infty}(\Rn)}
\end{align*}
for implicit constants independent of $z\in\overline{S}$, $j\geq0$ and $f\in C^{0}_{*}(\Rn)$.
\end{proof}

The following proposition is the main result of this subsection. 

\begin{proposition}\label{prop:inter}
Let $r>0$, $m\in\R$, $\delta\in[0,1]$, $l\in\N$ and $\kappa,\lambda\in\R$ be such that $\kappa+\lambda<r$. Then for each $c>0$ there exists an $M\geq0$ such that the following holds. Let $a\in C^{r}_{*}S^{m,l}_{1,\delta}$ be such that, for all $\eta\in\Rn$, one has $\supp(\F a(\cdot,\eta))\subseteq\{\xi\in\Rn\mid |\xi|\geq c|\eta|^{\delta}\}$. For $z\in\overline{S}$, set
\begin{equation}\label{eq:defaz}
a_{z}(x,\eta):=e^{(\kappa z+\lambda)^{2}}\lb \eta\rb^{-\delta (\kappa z+\lambda)}\big(\lb D\rb^{\kappa z+\lambda}a(\cdot,\eta)\big)(x).
\end{equation}
Then $a_{z}\in C^{r-\Real(\kappa z+\lambda)}_{*}S^{m,l}_{1,\delta}$ and $\|a_{z}\|_{C^{r-\Real(\kappa z+\lambda)}_{*}S^{m,l}_{1,\delta}}\leq M\|a\|_{C^{r}_{*}S^{m,l}_{1,\delta}}$.
\end{proposition}
The support assumption on $a$ is used in the proof to deduce the supremum bounds in part \eqref{it:symbol1} of Definition \ref{def:rough}. For the Zygmund bounds in part \eqref{it:symbol2}, which measure the spatial regularity of the symbol, this assumption is not needed.
\begin{proof}
Let $z\in\overline{S}$ and $\alpha\in\Z_{+}^{n}$ be such that $|\alpha|\leq l$. By \eqref{eq:Zyg1} and \eqref{eq:Zyg2}, $\lb D\rb^{r}:C^{r}_{*}(\Rn)\to C^{0}_{*}(\Rn)$ is invertible. Hence, by factorizing through $C^{0}_{*}(\Rn)$ and using \eqref{eq:Zyg1}, we obtain
\[
\sup\{\|e^{(\kappa z+\lambda)^{2}}\lb D\rb^{\kappa z+\lambda}\|_{\La(C^{r}_{*}(\Rn),C^{r-\Real(\kappa z+\lambda)}_{*}(\Rn))}\mid z\in\overline{S}\}<\infty.
\]
It follows that $e^{(\kappa z+\lambda)^{2}}\partial_{\eta}^{\alpha}\lb D\rb^{\kappa z+\lambda}a(\cdot,\eta)\in C^{r-\Real(\kappa z+\lambda)}_{*}(\Rn)$ for each $\eta\in\Rn$, with
\begin{align*}
&\|e^{(\kappa z+\lambda)^{2}}\partial_{\eta}^{\alpha}\lb D\rb^{\kappa z+\lambda}a(\cdot,\eta)\|_{C^{r-\Real(\kappa z+\lambda)}_{*}(\Rn)}\\
&\leq \|e^{(\kappa z+\lambda)^{2}}\lb D\rb^{\kappa z+\lambda}\|_{\La(C^{r}_{*}(\Rn),C^{r-\Real(\kappa z+\lambda)}_{*}(\Rn))} \|\partial_{\eta}^{\alpha}a(\cdot,\eta)\|_{C^{r}_{*}(\Rn)}\\
&\lesssim \|a\|_{C^{r}_{*}S^{m,l}_{1,\delta}} \lb\eta\rb^{m-|\alpha|+\delta r}\\
&=\|a\|_{C^{r}_{*}S^{m,l}_{1,\delta}}\lb\eta\rb^{m+\delta\Real(\kappa z+\lambda)-|\alpha|+\delta(r-\Real(\kappa z+\lambda))}.
\end{align*}
This suffices for part \eqref{it:symbol2} in Definition \ref{def:rough}, since $|\lb\eta\rb^{-\delta(\kappa z+\lambda)}|=\lb\eta\rb^{-\delta \Real(\kappa z+\lambda)}$.

Next, let $\eta\in\Rn$ be given, and let $k\geq1$ be such that $2^{k-1}\leq \lb\eta\rb\leq 2^{k+1}$. By assumption, for some $N\geq0$ independent of $\eta$ and $k$, one has
\[
a(x,\eta)=\sum_{j=0}^{\infty}\psi_{j}(D)a(\cdot,\eta)(x)=\sum_{j\geq \delta k-N}\psi_{j}(D) a(\cdot,\eta)(x)
\]
for all $x\in\Rn$.
Hence we can write
\begin{align*}
&\|e^{(\kappa z+\lambda)^{2}}\partial_{\eta}^{\alpha}\lb D\rb^{\kappa z+\lambda}a(\cdot,\eta)\|_{L^{\infty}(\Rn)}\\
&\leq \sum_{j\geq  \delta k-N}\|e^{(\kappa z+\lambda)^{2}}\lb D\rb^{\kappa z+\lambda}\psi_{j}(D)\partial_{\eta}^{\alpha}a(\cdot,\eta)\|_{L^{\infty}(\Rn)}.
\end{align*}
Now use \eqref{eq:Zyg3}, \eqref{eq:Zyg2}, and then \eqref{eq:Zyg3} again, to bound this by a multiple of
\begin{align*}
&\sum_{j\geq  \delta k-N}\|e^{(\kappa z+\lambda)^{2}}\lb D\rb^{\kappa z+\lambda}\psi_{j}(D)\partial_{\eta}^{\alpha}a(\cdot,\eta)\|_{C^{0}_{*}(\Rn)}\\
&\lesssim \sum_{j\geq  \delta k-N}\|\psi_{j}(D)\partial_{\eta}^{\alpha}a(\cdot,\eta)\|_{C^{\Real(\kappa z+\lambda)}_{*}(\Rn)}\\
&\eqsim \sum_{j\geq  \delta k-N}2^{j\Real(\kappa z+\lambda)}\|\psi_{j}(D)\partial_{\eta}^{\alpha}a(\cdot,\eta)\|_{L^{\infty}(\Rn)}.
\end{align*}
Finally, straightforward estimates allow one to bound this by a multiple of
\begin{align*}
&\sum_{j\geq \delta k-N}2^{j(\Real(\kappa z+\lambda)-r)}\|\partial_{\eta}^{\alpha}a(\cdot,\eta)\|_{C^{r}_{*}(\Rn)}\\
&\lesssim \|a\|_{C^{r}_{*}S^{m}_{1,\delta}}\sum_{j\geq \delta k-N}2^{-j(r-\Real(\kappa z+\lambda))}\lb\eta\rb^{m-|\alpha|+\delta r}\\
&\lesssim \|a\|_{C^{r}_{*}S^{m}_{1,\delta}}2^{-\delta k(r-\Real(\kappa z+\lambda))}\lb\eta\rb^{m-|\alpha|+\delta r}.
\end{align*}
This concludes the proof, since $2^{k}\eqsim \lb \eta\rb$.
\end{proof}

\subsection{Pseudodifferential operators}

In this subsection we collect some known results about the boundedness of certain classes of pseudodifferential operators. These preliminary results will be used in the proof of our main theorem.

 The following lemma is \cite[Lemma 3.2]{Rozendaal20}, which in turn is a straightforward consequence of \cite[Theorem 6.10]{HaPoRo20}. It will be used to deal with the smooth term in the decomposition of a rough symbol.

\begin{lemma}\label{lem:smoothpseudo}
Let $m\in\R$ and $a\in S^{m}_{1,1/2}$. Then 
\[
a(x,D):\HT^{s+m,p}_{FIO}(\Rn)\to\Hps
\]
for all $p\in[1,\infty]$ and $s\in\R$.
\end{lemma}

Next, we state a lemma about rough pseudodifferential operators acting on the classical function spaces $\HT^{s,p}(\Rn)$ from \eqref{eq:classical}. It is an extension of a result from \cite{Bourdaud82}. Without the additional parameter $l$ and the norm bounds in \eqref{it:pseudoLp1}, it is \cite[Lemma 3.1]{Rozendaal20}. The fact that one may add this parameter and obtain norm bounds follows either from abstract reasoning, or by keeping track of the constants in the proof of the relevant statement in \cite{Bourdaud82}.

\begin{lemma}\label{lem:pseudoLp}
Let $r>0$, $m\in\R$, $\delta\in[0,1)$ and $p\in[1,\infty]$. Then the following statements hold.
\begin{enumerate} 
\item\label{it:pseudoLp1} For each $-(1-\delta)r<s<r$, there exist an $l\in\N$ and an $M\geq0$ such that for each $a\in C^{r}_{*}S^{m,l}_{1,\delta}$ one has 
\begin{equation}\label{eq:pseudoLp}
a(x,D):\HT^{s+m,p}(\Rn)\to\HT^{s,p}(\Rn),
\end{equation}
with $\|a(x,D)\|_{\La(\HT^{s+m,p}(\Rn),\HT^{s,p}(\Rn))}\leq M\|a\|_{C^{r}_{*}S^{m,l}_{1,\delta}}$.
\item\label{it:pseudoLp2} For each $a\in \Hrinf S^{m}_{1,\delta}$, \eqref{eq:pseudoLp} also holds for $s=r$.
\item\label{it:pseudoLp3} For each $a\in\Hrinf S^{-\delta r}_{1,\delta}$ of the form $a=b^{\flat}_{\delta}$ for some $b\in\Hrinf (\Rn)$, \eqref{eq:pseudoLp} holds for all $-(1-\delta)r\leq s\leq r$, with $m=-\delta r$.
\end{enumerate}
\end{lemma}
We do not make any claims in \eqref{it:pseudoLp2} and \eqref{it:pseudoLp3} regarding norm bounds for $a(x,D)$ for the extremal values of $s$. Although these norm bounds are as one would expect, namely in terms of $\|a\|_{\Hrinf S^{m,l}_{1,\delta}}$ for sufficiently large $l$, we will not need them in the remainder. We also do not make any such claims in Proposition \ref{prop:pseudoloss} below.

As a consequence of Lemma \ref{lem:pseudoLp}, we directly obtain the following proposition, a version of \cite[Proposition 3.3]{Rozendaal20} which involves an additional parameter $l$ and norm bounds for the operators. Recall the definition of $s(p)$ from \eqref{eq:sp}.

\begin{proposition}\label{prop:pseudoloss}
Let $r>0$, $m\in\R$, $\delta\in[0,1)$ and $p\in[1,\infty]$. Then the following statements hold.
\begin{enumerate}
\item\label{it:pseudoloss1} For each $-(1-\delta)r-s(p)<s<r-s(p)$, there exist an $l\in\N$ and an $M\geq0$ such that for each $a\in C^{r}_{*}S^{m,l}_{1,\delta}$ one has
\begin{equation}\label{eq:pseudoloss}
a(x,D):\HT^{s+2s(p)+m,p}_{FIO}(\Rn)\to\Hps,
\end{equation}
with $\|a(x,D)\|_{\La(\HT^{s+2s(p)+m,p}_{FIO}(\Rn),\Hps)}\leq M\|a\|_{C^{r}_{*}S^{m,l}_{1,\delta}}$.
\item\label{it:pseudoloss2} For each $a\in \Hrinf S^{m}_{1,\delta}$, \eqref{eq:pseudoloss} also holds for $s=r-s(p)$. 
\item\label{it:pseudoloss3} For each $a\in \Hrinf S^{-\delta r}_{1,\delta}$ of the form $a=b^{\flat}_{\delta}$ for some $b\in\Hrinf (\Rn)$, \eqref{eq:pseudoloss} holds for all $-(1-\delta)r-s(p)\leq s\leq r-s(p)$, with $m=-\delta r$.
\end{enumerate}
\end{proposition}
\begin{proof}
Combine Lemma \ref{lem:pseudoLp} with the Sobolev embeddings for $\Hp$ from \eqref{eq:sobolev}:
\[
a(x,D):\HT^{s+2s(p)+m,p}_{FIO}(\Rn)\subseteq \HT^{s+s(p)+m,p}(\Rn)\to\HT^{s+s(p),p}(\Rn)\subseteq\Hps.\qedhere
\]
\end{proof}

\begin{remark}\label{rem:flatflat}
We will in fact not use the final statement in Proposition \ref{prop:pseudoloss} directly in what follows. Instead, in the proof of Theorem \ref{thm:main} we will use that \eqref{eq:pseudoloss} holds for all $-(1-\delta)r-s(p)\leq s\leq r-s(p)$, with $m=-\delta r$, if $a=((b^{\flat}_{\delta'})^{\flat}_{\delta'})^{\flat}_{\delta}$ for some $b\in\Hrinf(\Rn)$ and $\delta'\in[0,\delta]$. This was already noted in \cite[Remark 3.4]{Rozendaal20}, and to prove it one uses a straightforward modification of the proof of the final statement in Lemma \ref{lem:pseudoLp}.
\end{remark}

Our next proposition concerns the main result of \cite{Rozendaal20}. 

\begin{proposition}\label{prop:critical}
Let $r>0$, $m\in\R$, $p\in(1,\infty)$ and $s\in\R$. For $\veps>0$, set
\[
\tau:=\begin{cases}
0&\text{if }r>n-1,\\
\veps&\text{if }r=n-1,\\
2s(p)\big(1-\frac{r}{n-1}\big)&\text{if }r<n-1,
\end{cases}
\]
and
\[
\gamma:=\begin{cases}
\frac{1}{2}+\frac{2s(p)}{r}&\text{if }r\geq n-1,\\
\frac{1}{2}+\frac{2s(p)}{n-1}&\text{if }r<n-1.
\end{cases}
\]
Then, for each $c>0$, there exist $l\in\N$ and $M\geq0$ such that the following holds. For each $a\in C^{r}_{*}S^{m,l}_{1,1/2}$ such that
\begin{equation}\label{eq:support}
\supp(\F a(\cdot,\eta))\subseteq\{\xi\in\Rn\mid c|\eta|^{1/2}\leq |\xi|\leq  \tfrac{1}{16}(1+|\eta|)^{\gamma}\}
\end{equation}
for all $\eta\in\Rn$, one has $a(x,D):\HT^{s+m+\tau,p}_{FIO}(\Rn)\to \Hps$, with
\[
\|a(x,D)\|_{\La(\HT^{s+m+\tau,p}_{FIO}(\Rn),\Hps)}\leq M\|a\|_{C^{r}_{*}S^{m,l}_{1,1/2}}.
\]
\end{proposition}
\begin{proof}
Without the parameter $l$ and the norm bounds, the statement is in fact the heart of the proof of \cite[Theorem 4.1]{Rozendaal20}. Indeed, the first step of the proof of that theorem consists of reducing matters to this proposition. It also follows from the proof that one can add the parameter $l$ and obtain norm bounds, as is noted in \cite[Remark 4.2]{Rozendaal20}.
\end{proof}

\section{Main result}\label{sec:main}

We are now ready to state and prove our main result. In the same manner as in Propositions \ref{prop:pseudoloss} and \ref{prop:critical}, one can add a parameter $l$ and obtain norm bounds for the associated operators, but for simplicity we do not include this additional information in the statement.

\begin{theorem}\label{thm:main}
Let $r>0$, $m\in\R$, $p\in(1,\infty)$ and $a\in C^{r}_{*}S^{m}_{1,1/2}$. For $\veps\in(0,r/2]$, set
\[
\sigma:=\begin{cases}
0&\text{if }2s(p)<r/2,\\
2s(p)-r/2+\veps&\text{if }2s(p)\geq r/2.
\end{cases}
\]
Then 
\begin{equation}\label{eq:main}
a(x,D):\HT^{s+m+\sigma,p}_{FIO}(\Rn)\to \Hps
\end{equation}
for $-r/2+s(p)-\sigma<s<r-s(p)$.
If $a\in\Hrinf S^{m}_{1,1/2}$, then \eqref{eq:main} also holds for $s=r-s(p)$. If $a=b^{\flat}_{1/2}$ for some $b\in \Hrinf (\Rn)$, then \eqref{eq:main} holds for all $-r/2+s(p)-\sigma\leq s\leq r-s(p)$, with $m=-r/2$.
\end{theorem}
\begin{proof}
Firstly, by replacing $a(x,D)$ by $a(x,D)\lb D\rb^{-m}$, we may assume that $m=0$, for notational simplicity. For the final statement this requires replacing $a(x,D)$ by $b^{\flat}_{1/2}(x,D)\lb D\rb^{r/2}$.

\subsubsection{Symbol smoothing}

Next, we use the symbol smoothing procedure to remove some of the frequencies of $a$. More precisely, for 
\[
\beta:=\frac{1}{2}+\frac{2s(p)-\sigma}{r},
\]
we claim that it suffices to prove the following statement. If $a\in C^{r}_{*}S^{0}_{1,1/2}$ has the additional property that, for some $c>0$ and all $\xi,\eta\in\Rn$, one has
\begin{equation}\label{eq:supportcondition1}
\supp(\F a(\cdot,\eta))\subseteq\{\xi\in\Rn\mid c|\eta|^{1/2}\leq |\xi|\leq  \tfrac{1}{16}(1+|\eta|)^{\beta}\},
\end{equation}
then \eqref{eq:main} holds for all $s\in\R$.

To prove this claim, let a general $a\in C^{r}_{*}S^{0}_{1,1/2}$ be given. Note that $1/2\leq \beta<1$. We apply the symbol smoothing procedure twice, to write
\begin{equation}\label{eq:doubledec}
a=a_{1/2}^{\sharp}+a^{\flat}_{1/2}=a_{1/2}^{\sharp}+(a^{\flat}_{1/2})^{\sharp}_{\beta}+(a^{\flat}_{1/2})^{\flat}_{\beta}.
\end{equation}
By Lemma \ref{lem:smoothing}, one has $a^{\sharp}_{1/2}\in S^{0}_{1,1/2}$ and $(a^{\flat}_{1/2})^{\flat}_{\beta}\in C^{r}_{*}S^{-(\beta-1/2)r}_{1,\beta}$, with $(a^{\flat}_{1/2})^{\flat}_{\beta}\in \Hrinf S^{-(\beta-1/2)r}_{1,\beta}$ if $a\in \Hrinf S^{0}_{1,1/2}$. Hence Lemma \ref{lem:smoothpseudo} yields 
\begin{equation}\label{eq:sharp}
a^{\sharp}_{1/2}(x,D):\HT^{s+\sigma,p}_{FIO}(\Rn)\subseteq \Hps\to \Hps
\end{equation}
for all $s\in\R$. Also, Proposition \ref{prop:pseudoloss} yields
\begin{equation}\label{eq:doubleflat}
(a^{\flat}_{1/2})^{\flat}_{\beta}(x,D):\HT^{s+\sigma,p}_{FIO}(\Rn)=\HT^{s+2s(p)-(\beta-1/2)r,p}_{FIO}(\Rn)\to\Hps
\end{equation}
for all 
\[
-\frac{r}{2}+s(p)-\sigma=-(1-\beta)r-s(p)<s<r-s(p),
\] 
and also for $s=r-s(p)$ if $a\in \Hrinf S^{0}_{1,1/2}$. Finally, if $a(x,D)=b_{1/2}^{\flat}(x,D)\lb D\rb^{r/2}$ for some $b\in \Hrinf (\Rn)$, then Remark \ref{rem:flatflat} shows that \eqref{eq:doubleflat} also holds for $s=-r/2+s(p)-\sigma$. 

By combining \eqref{eq:doubledec}, \eqref{eq:sharp} and \eqref{eq:doubleflat}, we see that it suffices to show in the rest of the proof that 
\[
(a^{\flat}_{1/2})^{\sharp}_{\beta}(x,D):\HT^{s+\sigma,p}_{FIO}(\Rn)\to \Hps
\]
for all $s\in\R$. Now, by Lemma \ref{lem:smoothing} and \eqref{eq:symbolinc}, one has $a^{\flat}_{1/2}\in C^{r}_{*}S^{0}_{1,1/2}$ and $(a^{\flat}_{1/2})^{\flat}_{\beta}\in C^{r}_{*}S^{-(\beta-1/2)r}_{1,\beta}\subseteq C^{r}_{*}S^{0}_{1,1/2}$, so $(a^{\flat}_{1/2})^{\sharp}_{\beta}\in C^{r}_{*}S^{0}_{1,1/2}$ as well. Moreover, a straightforward computation shows that, for $\ph$ with sufficiently small support (independent of $a$), $(a^{\flat}_{1/2})^{\sharp}_{\beta}$ has the property in \eqref{eq:supportcondition1}. 

This proves the claim, and the remainder of the proof will be dedicated to proving \eqref{eq:main} for $s\in\R$ and $a\in C^{r}_{*}S^{0}_{1,1/2}$ satisfying \eqref{eq:supportcondition1}. To do so we will use interpolation of analytic families of operators, combined with Propositions \ref{prop:pseudoloss} and \ref{prop:critical}.

\subsubsection{Interpolation setup}

We prepare for the interpolation procedure by introducing some parameters. We may suppose that $p\neq 2$, since otherwise Proposition \ref{prop:pseudoloss} directly yields the required statement. In fact, we will assume that $p\in(1,2)$. The case where $p\in(2,\infty)$ is dealt with in an analogous\footnote{Note that we cannot rely directly on duality here, since the operators $a(x,D)$ do not behave well under taking adjoints, especially for rough symbols $a$.} manner. 

Let $\delta\in(0,r)$ be such that $1+\delta<p$, set $r_{1}:=\delta$, and let $\theta\in(0,1)$, $r_{0}>0$ and $t\in\R$ be such that 
\[
\frac{1}{p}=\frac{1-\theta}{1+\delta}+\frac{\theta}{2},\quad\quad r=(1-\theta)r_{0}+\theta r_{1},\quad\text{and}\quad s=(1-\theta)t.
\]
Then 
\[
1-\theta=\frac{\frac{1}{p}-\frac{1}{2}}{\frac{1}{1+\delta}-\frac{1}{2}}=\frac{2s(p)}{(n-1)(\frac{1}{1+\delta}-\frac{1}{2})}=\frac{2s(p)}{2s(1+\delta)}
\]
and
\[
r_{0}=\frac{r-\theta\delta}{1-\theta}=(n-1)(r-\theta\delta)\frac{\frac{1}{1+\delta}-\frac{1}{2}}{2s(p)}=(r-\theta\delta)\frac{2s(1+\delta)}{2s(p)}.
\]
Next, let $\tau$ and $\gamma$ be as in Proposition \ref{prop:critical} with $p$ replaced by $1+\delta$ and $r$ replaced by $r_{0}$. By the choice of parameters, we then have
\begin{equation}\label{eq:tau}
\begin{aligned}
(1-\theta)\tau&=\begin{cases}
0&\text{if }r_{0}>n-1,\\
(1-\theta)\veps&\text{if }r_{0}=n-1,\\
(1-\theta)2s(1+\delta)\big(1-\frac{r_{0}}{n-1}\big)&\text{if }r_{0}>n-1,
\end{cases}\\
&=\begin{cases}
0&\ \ \text{if }(\frac{1}{1+\delta}-\frac{1}{2})^{-1}2s(p)<r-\theta\delta,\\
(1-\theta)\veps&\ \ \text{if }(\frac{1}{1+\delta}-\frac{1}{2})^{-1}2s(p)=r-\theta\delta,\\
2s(p)-(r-\theta\delta)(\frac{1}{1+\delta}-\frac{1}{2})&\ \ \text{if }(\frac{1}{1+\delta}-\frac{1}{2})^{-1}2s(p)>r-\theta\delta,
\end{cases}
\end{aligned}
\end{equation}
and
\begin{align*}
\gamma&=\begin{cases}
\frac{1}{2}+\frac{2s(1+\delta)}{r_{0}}&\text{if }r_{0}\geq n-1,\\
\frac{1}{2}+\frac{2s(1+\delta)}{n-1}&\text{if }r_{0}<n-1,
\end{cases}\\
&=\begin{cases}
\frac{1}{2}+\frac{2s(p)}{r-\theta\delta}&\quad\text{if }r_{0}\geq n-1,\\
\frac{1}{1+\delta}&\quad\text{if }r_{0}<n-1.
\end{cases}
\end{align*}
In the remainder we will choose $\delta$ sufficiently small such that $\beta\leq \gamma$ holds, which is possible since $\sigma\geq0$ and $\beta<1$. 

Next, we introduce an associated collection of symbols and derive some of their properties. Let $S=\{z\in\C\mid 0<\Real(z)< 1\}$ be as in \eqref{eq:S}. For $z\in\overline{S}$ and $x,\eta\in\Rn$ write, as in Proposition \ref{prop:inter},
\[
a_{z}(x,\eta):=e^{(\kappa z+\lambda)^{2}}\lb\eta\rb^{-(\kappa z+\lambda)/2}\big(\lb D\rb^{\kappa z+\lambda}a(\cdot,\eta)\big)(x),
\]
where $\kappa:=r_{0}-r_{1}$ and $\lambda:=r-r_{0}$. It then follows from the definitions of $r_{0}$, $r_{1}$ and $\theta$ that $a_{\theta}=a$. Moreover, \eqref{eq:supportcondition1} and Proposition \ref{prop:inter} imply that for each $l\in\N$ there exists an $M_{l}\geq0$ such that $a_{z}\in C^{r-\Real(\kappa z+\lambda)}_{*}S^{0}_{1,1/2}$ for each $z\in\overline{S}$, with
\begin{equation}\label{eq:bznorm}
\sup\{\|a_{z}\|_{C^{r-\Real(\kappa z+\lambda)}_{*}S^{0,l}_{1,1/2}}\mid z\in\overline{S}\}\leq M_{l}\|a\|_{C^{r}_{*}S^{0,l}_{1,1/2}} <\infty.
\end{equation}
Also, since $\beta\leq \gamma$, another application of \eqref{eq:supportcondition1} shows that
\begin{equation}\label{eq:supportbz}
\supp(\F a_{z}(\cdot,\eta))\subseteq\{\xi\in\Rn\mid c|\eta|^{1/2}\leq |\xi|\leq  \tfrac{1}{16}(1+|\eta|)^{\gamma}\}.
\end{equation} 

\subsubsection{Interpolation}

We are now ready to use complex interpolation of analytic families of operators. Let 
\[
F:\overline{S}\to\La(L^{2}(\Rn))
\]
be given by 
\[
F(z):=a_{z}(x,D)\quad(z\in\overline{S}).
\] 
This is well defined by \eqref{eq:bznorm} and Proposition \ref{prop:pseudoloss}, since $\HT^{2}_{FIO}(\Rn)=L^{2}(\Rn)$. In fact, $C^{r-\Real(\kappa z+\lambda)}_{*}S^{0,l}_{1,1/2}$ embeds contractively into $ C^{r-(\kappa+\lambda)}_{*}S^{0,l}_{1,1/2}$ for each $l\in\N$ (see \eqref{eq:Zygmund1}). Hence \eqref{eq:bznorm} yields $\sup\{\|a_{z}\|_{C^{r-(\kappa+\lambda)}_{*}S^{0,l}_{1,1/2}}\mid z\in\overline{S}\}<\infty$, and
\begin{equation}\label{eq:L2bound}
\sup\{\|F(z)\|_{\La(L^{2}(\Rn))}\mid z\in\overline{S}\}<\infty
\end{equation}
by Lemma \ref{lem:pseudoLp}.

Next, we claim that for each $f\in\Sw(\Rn)$ with compact Fourier support, the following conditions hold:
\begin{enumerate}
\item\label{it:interproof1} One has $F(i\tau):\HT^{t+\tau,1+\delta}_{FIO}(\Rn)\to \HT^{t,1+\delta}_{FIO}(\Rn)$ for each $\tau\in\R$, and
\[
\sup_{\tau\in\R} \|F(i\tau)\|_{\La(\HT^{t+\tau,1+\delta}_{FIO}(\Rn),\HT^{t,1+\delta}_{FIO}(\Rn))}<\infty.
\]
\item\label{it:interproof2} The map $z\mapsto F(z)f$ is continuous on $\overline{S}$ with values in $L^{2}(\Rn)\cap \HT^{t,1+\delta}_{FIO}(\Rn)$. 
\item\label{it:interproof3} The map $z\mapsto F(z)f$ is holomorphic on $S$ with values in $L^{2}(\Rn)$.
\end{enumerate}
For the moment, suppose that we have proved this claim. Then, using also \eqref{eq:L2bound} and the remark after Proposition \ref{prop:interan}, we can apply that proposition, since the Schwartz functions with compact Fourier support lie dense in $\HT^{t+\tau,1+\delta}_{FIO}(\Rn)\cap L^{2}(\Rn)$. Then 
\[
a(x,D)=F(\theta):\HT^{s+(1-\theta)\tau,p}_{FIO}(\Rn)\to\Hps,
\]
as follows from the choice of the relevant parameters.
By using \eqref{eq:tau} and splitting into several cases, one can check that for sufficiently small $\delta$ this proves the required statement. Note that, when $4s(p)=r$, this requires relying on the final case in \eqref{eq:tau} for small $\delta$, since 
\[
2s(p)-\big(\tfrac{1}{1+\delta}-\tfrac{1}{2}\big)(r-\theta\delta)\to0
\]
as $\delta\to0$. Hence it only remains to prove \eqref{it:interproof1}, \eqref{it:interproof2} and \eqref{it:interproof3}.

\subsubsection{Condition \eqref{it:interproof1}}

By Proposition \ref{prop:inter}, one has $a_{i\tau}\in C^{r-\Real(\kappa i\tau+\lambda)}_{*}S^{0,l}_{1,1/2}=C^{r_{0}}_{*}S^{0,l}_{1,1/2}$
for all $\tau\in\R$ and $l\in\N$, with $\sup\{\|a_{i\tau}\|_{C^{r_{0}}_{*}S^{0,l}_{1,1/2}}\mid \tau\in\R\}<\infty$. By combining this with \eqref{eq:supportbz} and Proposition \ref{prop:critical}, we arrive at the desired conclusion.

\subsubsection{Condition \eqref{it:interproof2}}

We will in fact show that $z\mapsto F(z)f$ is continuous on $\overline{S}$ as an $\HT^{v,u}_{FIO}(\Rn)$-valued map for all $u\in(1,\infty)$ and $v\in\R$, thereby simultaneously covering the cases where $u=2$ and $v=0$ (recall that $L^{2}(\Rn)=\HT^{2}_{FIO}(\Rn)$), and $u=1+\delta$ and $v=t$. 

To this end, first note that there exists a $\psi\in C^{\infty}_{c}(\Rn)$ such that $a_{z}(x,D)f=a_{z}(x,D)\psi(D)f$ for each $z\in\overline{S}$, by the assumption of compact Fourier support on $f$. Now, by \eqref{eq:supportcondition1}, there exists a compact $K\subseteq \Rn$ such that 
\begin{equation}\label{eq:supp3}
\supp(\F a(\cdot,\eta)\psi(\eta))\subseteq K
\end{equation} 
for all $\eta\in\Rn$, and $a(x,\eta)\psi(\eta)=0$ for all $x\in\Rn$ if $\eta\notin K$. This in turn implies that the symbol $a\psi$ is in fact an element of $C^{\rho}_{*}S^{\sigma}_{1,1/2}$ for all $\rho>0$ and $\sigma\in\R$, as is straightforward to check. Hence $a_{z}\psi\in C^{\rho'}_{*}S^{\sigma,l}_{1,1/2}$ for all $\rho'>0$, $\sigma\in\R$, $l\in\N$ and $z\in\overline{S}$, and $\sup\{\|a_{z}\psi\|_{C^{\rho'}_{*}S^{\sigma,l}_{1,1/2}}\mid z\in\overline{S}\}<\infty$, by Proposition \ref{prop:inter}. In particular, 
\begin{equation}\label{eq:symbolsup}
\sup\{|\partial_{x}^{\alpha}\partial_{\eta}^{\beta}a_{z}(x,\eta)\psi(\eta)|\mid z\in\overline{S},x,\eta\in\Rn\}<\infty
\end{equation}
for all $\alpha,\beta\in\Z_{+}^{n}$, by \eqref{eq:Zygmund3}.

Now let $(z_{j})_{j=0}^{\infty}\subseteq \overline{S}$ and $z\in \overline{S}$ be such that $z_{j}\to z$ as $j\to\infty$. By \eqref{eq:HpFIOnorm}, we want to show that
\begin{equation}\label{eq:Lqtoshow1}
\int_{S^{n-1}}\|\lb D\rb^{v}\ph_{\w}(D)(a_{z_{j}}(x,D)-a_{z}(x,D))\psi(D)f\|_{L^{u}(\Rn)}^{u}\ud\w\to 0
\end{equation}
and
\begin{equation}\label{eq:Lqtoshow2}
\|q(D)(a_{z_{j}}(x,D)-a_{z}(x,D))\psi(D)f\|_{L^{u}(\Rn)}\to 0
\end{equation}
as $j\to\infty$. We will prove \eqref{eq:Lqtoshow1}, with the argument for \eqref{eq:Lqtoshow2} being similar but simpler. The proof mainly consists of applying the dominated convergence theorem several times, with some minor subtleties.

By \eqref{eq:Lqmap2}, it suffices to show that 
\[
\|(a_{z_{j}}(x,D)-a_{z}(x,D))\psi(D)f\|_{W^{N',u}(\Rn)}\to0
\]
for some $N'\in\N$ with $N'\geq 2n+v$, i.e., that
\begin{equation}\label{eq:Lqtoshow3}
\int_{\Rn}\big|\partial^{\alpha}_{x}\big((a_{z_{j}}(x,D)-a_{z}(x,D))\psi(D)f\big)(x)\big|^{u}\ud x\to0
\end{equation}
for each $\alpha\in\Z_{+}^{n}$ with $|\alpha|\leq N'$. Note that 
\begin{align*}
&\lb x\rb^{2n}|\partial_{x}^{\alpha}\big((a_{z_{j}}(x,D)-a_{z}(x,D))\psi(D)f\big)(x)|\\
&=(2\pi)^{-n}\big|(1+|x|^{2})^{n}\partial_{x}^{\alpha}\int_{\Rn}e^{ix\eta}(a_{z_{j}}(x,\eta)-a_{z}(x,\eta))\psi(\eta)\wh{f}(\eta)\ud\eta\big|
\end{align*}
for each $x\in\Rn$. 
It follows by integrating by parts, using \eqref{eq:symbolsup} and that $\psi,f\in\Sw(\Rn)$, that the latter quantity is uniformly bounded in $x$ and $j$. Hence, by the dominated convergence theorem and because $x\mapsto \lb x\rb^{-2nu}$ is integrable, for \eqref{eq:Lqtoshow3} it suffices to show that
\[
\partial^{\alpha}_{x}\big(a_{z_{j}}(x,D)\psi(D)f\big)(x)\to \partial^{\alpha}_{x}\big(a_{z}(x,D)\psi(D)f\big)(x)
\]
for each $x\in\Rn$, as $j\to\infty$. 

Again, one has
\[
\partial^{\alpha}_{x}\big(a_{z_{j}}(x,D)\psi(D)f\big)(x)=(2\pi)^{-n}\partial_{x}^{\alpha}\int_{\Rn}e^{ix\eta}a_{z_{j}}(x,\eta)\psi(\eta)\wh{f}(\eta)\ud\eta.
\]
Hence, because $f\in\Sw(\Rn)$, one can use \eqref{eq:symbolsup} and Leibniz' rule, while possibly replacing $\alpha$ by a new value, to reduce to showing that $\partial^{\alpha}_{x}a_{z_{j}}(x,\eta)\psi(\eta)\to \partial^{\alpha}_{x}a_{z}(x,\eta)\psi(\eta)$ as $j\to\infty$, for all $x,\eta\in\Rn$. 

For this final part of the argument we cannot apply the dominated convergence theorem to $\F a_{z_{j}}(\cdot,\eta)$ directly, since it is not clear whether the distribution $\F a(\cdot,\eta)$ coincides with a locally bounded function. However, by \eqref{eq:supp3}, there exists a $\wt{\psi}\in C^{\infty}_{c}(\Rn)$ such that 
\[
\partial_{x}^{\alpha}a_{z_{j}}(x,\eta)\psi(\eta)=\wt{\psi}(D)a_{z_{j}}(\cdot,\eta)(x)\psi(\eta)
\]
for all $j\geq0$. Set $\wt{\psi}_{j}(\xi):=\lb \xi\rb^{\kappa z_{j}+\lambda}\wt{\psi}(\xi)$ for $\xi\in\Rn$. Then
\[
\partial_{x}^{\alpha}a_{z_{j}}(x,\eta)\psi(\eta)=e^{(\kappa z_{j}+\lambda)^{2}}\int_{\Rn}\F^{-1}(\wt{\psi}_{j})(x-y)a(y,\eta)\ud y\,\lb\eta\rb^{-(\kappa z_{j}+\lambda)/2}\psi(\eta),
\] 
and another two applications of the dominated convergence theorem reduce matters to the pointwise convergence
\[
e^{(\kappa z_{j}+\lambda)^{2}}\lb\eta\rb^{-(\kappa z_{j}+\lambda)/2}\lb \xi\rb^{\kappa z_{j}+\lambda}\to e^{(\kappa z+\lambda)^{2}}\lb\eta\rb^{-(\kappa z+\lambda)/2}\lb \xi\rb^{\kappa z+\lambda}
\]
as $j\to\infty$. 

\subsubsection{Condition \eqref{it:interproof3}}

The proof that this condition is satisfied is analogous to that of condition \eqref{it:interproof2}, and in fact the statement also holds for $z\mapsto F(z)f$ as a map with values in $\HT^{v,u}_{FIO}(\Rn)$ for any $u\in(1,\infty)$ and $v\in\R$. For fixed $\xi\in\Rn$, note that
\[
\lb \xi\rb^{\kappa z+\lambda}=(\lb\xi\rb^{\kappa})^{z}\lb \xi\rb^{\lambda}=\sum_{k=0}^{\infty}\frac{(\kappa\log\lb\xi\rb)^{k}}{k!}z^{k}\lb \xi\rb^{\lambda}
\]
for all $z\in\C$, and similarly for $\lb\eta\rb^{-(\kappa z+\lambda)/2}$ for $\eta\in\Rn$. Now, by applying the reasoning from the proof of condition \eqref{it:interproof2} in reverse, one sees that the resulting power series for $a_{z}(x,D)\psi(D)f$ converges in $\HT^{v,u}_{FIO}(\Rn)$. 
\end{proof}

\begin{remark}\label{rem:p1}
The restriction that $p\in(1,\infty)$ in Theorem \ref{thm:main} arises from Proposition \ref{prop:critical}. The proof of that proposition in \cite{Rozendaal20} uses techniques that do not appear to translate directly to the cases where $p=1$ or $p=\infty$ (see \cite[Remark 4.3]{Rozendaal20}). We expect that the statement of Theorem \ref{thm:main} itself is also valid for $p=1$ and $p=\infty$, but the interpolation techniques of the present article do not allow one to deal with these extremal values.
\end{remark}

\begin{remark}\label{rem:restrictedp}
A restriction of some sort on $p$ seems reasonable in Theorem \ref{thm:intro}. Indeed, for small $r$ and for $p$ far away from $2$, even rough multiplication operators do not preserve the classical function spaces in the Sobolev embeddings for $\Hp$ from \eqref{eq:Sobolev1}, and these embeddings are sharp. In the present article we will not address the question whether the specific condition on $p$ in Theorem \ref{thm:intro} is optimal.
\end{remark}

\begin{remark}\label{rem:otherinter}
One can also prove that there exists an open interval of $p$ around $2$ such that $a(x,D)$ acts boundedly on suitable Sobolev spaces over $\Hp$, for any $r>0$ and $a\in C^{r}_{*}S^{0}_{1,1/2}$, without relying on Proposition \ref{prop:critical}. Indeed, in Lemma \ref{lem:smoothpseudo} one in fact does not need the symbol $a$ to be infinitely smooth, and instead $C^{N}_{*}S^{0}_{1,1/2}$ regularity for a large $N=N(n)>0$ suffices. One can then interpolate with Proposition \ref{prop:pseudoloss}, in the same manner as before, to obtain such an interval around $2$. However, the resulting interval would be substantially smaller than that in Theorem \ref{thm:main}. Moreover, the proof in \cite{Rozendaal20} of Lemma \ref{lem:smoothpseudo}, which relies on \cite[Theorem 6.10]{HaPoRo20} and on an equivalent characterization of $\Hp$ from \cite{Rozendaal21}, is no simpler than that of Proposition \ref{prop:critical}.
\end{remark}

By combining Theorem \ref{thm:main} and Lemma \ref{lem:smoothing} we obtain the following corollary, for pseudodifferential operators with $C^{r}_{*}S^{m}_{1,\delta}$ symbols for general $\delta\in[0,1/2]$. 

\begin{corollary}\label{cor:alldelta}
Let $r>0$, $m\in\R$, $\delta\in[0,1/2]$, $p\in(1,\infty)$ and $a\in C^{r}_{*}S^{m}_{1,\delta}$. For $\veps\in(0,r/2]$, set
\[
\rho:=\begin{cases}
0&\text{if }2s(p)<(1-\delta)r,\\
2s(p)-(1-\delta)r+\veps&\text{if }2s(p)\geq(1-\delta)r.
\end{cases}
\]
Then 
\begin{equation}\label{eq:cormain}
a(x,D):\HT^{s+m+\rho,p}_{FIO}(\Rn)\to \Hps
\end{equation}
for $-r/2+s(p)-\sigma<s<r-s(p)$, where $\sigma$ is as in Theorem \ref{thm:main}.
If $a\in\Hrinf S^{m}_{1,\delta}$, then \eqref{eq:cormain} also holds for $s=r-s(p)$. If $a\in \Hrinf (\Rn)$, then \eqref{eq:cormain} holds for all $-r/2-s(p)-\sigma\leq s\leq r-s(p)$, with $m=\delta=0$.
\end{corollary}
Note that $\rho=\max(0,\sigma-(1/2-\delta)r)$.
\begin{proof}
Write $a=a^{\sharp}_{1/2}+a^{\flat}_{1/2}$. Then, by Lemma \ref{lem:smoothing}, one has $a^{\sharp}_{1/2}\in S^{m}_{1,1/2}$ and $a^{\flat}_{1/2}\in C^{r}_{*}S^{m-(1/2-\delta)r}_{1,1/2}$. Moreover, if $a\in \Hrinf S^{m}_{1,\delta}$ then $a^{\flat}_{1/2}\in \Hrinf S^{m-(1/2-\delta)r}_{1,1/2}$. Now Lemma \ref{lem:smoothpseudo} implies that
\[
a^{\sharp}_{1/2}(x,D):\HT^{s+m+\rho,p}_{FIO}(\Rn)\subseteq \HT^{s+m,p}_{FIO}(\Rn)\to\Hps
\]
for all $s\in\R$. Moreover, by Theorem \ref{thm:main} one has
\[
a^{\flat}_{1/2}(x,D):\HT^{s+m+\rho,p}_{FIO}(\Rn)\subseteq \HT^{s+m+\sigma-(1/2-\delta)r,p}_{FIO}(\Rn)\to\Hps
\]
for $s$ as in the statement of the corollary.
\end{proof}

\begin{remark}\label{rem:mult}
Corollary \ref{cor:alldelta} applies in particular to multiplication with $C^{r}_{*}(\Rn)$ or $\Hrinf(\Rn)$ functions. For example, for $r>0$, the multiplication operator with symbol $a\in C^{r}_{*}(\Rn)$ is bounded on $\Hps$ if $r>2s(p)$ and $-r/2+s(p)<s<r-s(p)$. If $a\in  \Hrinf (\Rn)$ then one may also let $s=-r/2+s(p)$ and $s=r-s(p)$. See also \cite{Frey-Portal20} for results about multiplication operators on adapted Hardy spaces related to $\Hp$.
\end{remark}

\begin{remark}\label{rem:comparison}
It is illustrative to compare Theorem \ref{thm:main} to \cite[Theorem 4.1]{Rozendaal20}. There it is shown that, under the assumptions of Theorem \ref{thm:main}, one has 
\begin{equation}\label{eq:taumapping}
a(x,D):\HT^{s+m+\tau,p}_{FIO}(\Rn)\to\Hps
\end{equation}
for $\tau$ as in Proposition \ref{prop:critical}. For $r>n-1$ one has $\sigma=\tau$, but for $r\leq n-1$ and $p\in(1,\infty)\setminus\{2\}$ one has $\sigma< \tau$. More precisely, if $r<n-1$ then 
\[
\tau-\sigma=
\begin{cases}
(n-1-r)|\tfrac{1}{2}-\tfrac{1}{p}|&\text{if }2s(p)<r/2,\\
r(\tfrac{1}{2}-|\tfrac{1}{2}-\tfrac{1}{p}|)-\veps&\text{if }2s(p)\geq r/2,
\end{cases}
\]
and $\tau-\sigma=\veps$ if $r=n-1$. On the other hand, $\tau-\sigma\to0$ as $p\to1$ or $p\to\infty$. 

It should also be noted that the Sobolev interval $-r/2+s(p)-\sigma<s<r/2$ in Theorem \ref{thm:main} is, for some values of $r$ and $p$, smaller than the one in \cite[Theorem 4.1]{Rozendaal20}. There \eqref{eq:taumapping} is shown to hold for $-(1-\gamma)r-s(p)<s<r-s(p)$ if $a\in C^{r}_{*}S^{m}_{1,1/2}$, with $\gamma$ as in Proposition \ref{prop:critical}, and with similar endpoint statements as in Theorem \ref{thm:main}. On the other hand, one can of course enlarge the Sobolev interval in Theorem \ref{thm:main} at the cost of additional regularity, using the trivial embedding $\HT^{s+m+t,p}_{FIO}(\Rn)\subseteq \HT^{s+m+\sigma,p}_{FIO}(\Rn)$ for $t>\sigma$.

Corollary \ref{cor:alldelta} can be compared to \cite[Corollary 4.5]{Rozendaal20} in a similar manner.
\end{remark}

\section*{Acknowledgments}

The author would like to thank Andrew Hassell for various helpful conversations about this article. He is also grateful to the anonymous referees for their careful reading of the manuscript, and for useful comments and suggestions.

\bibliographystyle{plain}
\bibliography{Bibliography}

\end{document}